\documentclass[11pt,twoside,leqno]{article}
\usepackage{amsmath,amsthm,amssymb}
\numberwithin{equation}{section}
\usepackage[dvips]{graphicx,color,psfrag}

\makeatletter
\newcommand{\figcaption}[1]{\def\@captype{figure}\caption{#1}}
\newcommand{\tblcaption}[1]{\def\@captype{table}\caption{#1}}
\makeatother

\newcommand{\mG}{\mathbf{\G}}

\makeatletter

\ifcase\@ptsize
\def\rpkern{\mathchoice{\kern-1.45em}{\kern-1.11em}{}{}}%
\def\grpkern{\mathchoice{\kern-1.013em}{\kern-0.825em}{}{}}%
\or
\def\rpkern{\mathchoice{\kern-1.44em}{\kern-1.11em}{}{}}%
\def\grpkern{\mathchoice{\kern-1.00em}{\kern-0.81em}{}{}}%
\or
\def\rpkern{\mathchoice{\kern-1.472em}{\kern-1.14em}{}{}}%
\def\grpkern{\mathchoice{\kern-1.00em}{\kern-0.815em}{}{}}%
\fi

\def\minibullet{\mathchoice%
{\raise0.2ex\hbox{$\scriptstyle\bullet$}}%
{\raise0.26ex\hbox{$\scriptscriptstyle\bullet$}}{}{}}
\def\butabullet{\mathchoice%
{\raise0.8ex\hbox{$\scriptstyle\bullet$}{\kern-0.365em}%
\lower0.4ex\hbox{$\scriptstyle\bullet$}}%
{\raise0.75ex\hbox{$\scriptscriptstyle\bullet$}{\kern-0.335em}%
\lower0.25ex\hbox{$\scriptscriptstyle\bullet$}}{}{}}

\def\regprod{\operatorname{\coprod\rpkern\prod}\displaylimits}

\def\customprod#1#2%
{\operatorname{\coprod\kern#2{#1}\kern#2\prod}\displaylimits}

\newcommand{\bA}{\mathbb{A}}

\newcommand{\bC}{\mathbb{C}}

\newcommand{\bN}{\mathbb{N}}

\newcommand{\bQ}{\mathbb{Q}}
\newcommand{\bR}{\mathbb{R}}

\newcommand{\bZ}{\mathbb{Z}}

\newcommand{\cO}{\mathcal{O}}

\newcommand{\ff}{\mathfrak{f}}

\newcommand{\fp}{\mathfrak{p}}

\renewcommand{\a}{\alpha}

\renewcommand{\d}{\delta}
\newcommand{\e}{\varepsilon}
\newcommand{\z}{\zeta}

\newcommand{\m}{\mu}

\renewcommand{\r}{\rho}

\renewcommand{\c}{\chi}
\newcommand{\vp}{\varphi}

\newcommand{\G}{\Gamma}

\renewcommand{\L}{\Lambda}

\renewcommand{\Re}{\mathrm{Re}\,}
\renewcommand{\Im}{\mathrm{Im}\,}

\newcommand{\GL}{\mathrm{GL}}

\newcommand{\p}{\partial}

\newcommand{\bsym}{\boldsymbol}

\newcommand{\boldtitle}[1]{\title{\bfseries #1}}

\newenvironment{MSC}{%
\smallbreak
\noindent \textbf{2010\ Mathematics Subject Classification\,:}}

\newenvironment{keywords}{%
\noindent\textbf{Key words and phrases\,:}\itshape}

\newenvironment{Acknowledgement}{%
\noindent\textit{Acknowledgement.}}

\theoremstyle{theorem}
\newtheorem*{multitheorem}{\variable@name}

\theoremstyle{definition}
\newcommand{\variable@name}{Theorem}
\newtheorem*{multiproclaim}{\variable@name}

\theoremstyle{plain}
\newtheorem{thm}{Theorem}[section]
\newtheorem{prop}[thm]{Proposition}
\newtheorem{lem}[thm]{Lemma}
\newtheorem{cor}[thm]{Corollary}

\theoremstyle{definition}

\newtheorem{example}[thm]{Example}
\newtheorem{remark}[thm]{Remark}

\boldtitle{Higher regularizations for zeros of cuspidal automorphic $L$-functions of $\GL_d$}
\author{
 Masato WAKAYAMA\thanks{Partially supported by Grant-in-Aid for Scientific Research (B) No. 21340011.} 
 \ and
Yoshinori YAMASAKI\thanks{Partially supported by Grant-in-Aid for Young Scientists (B) No. 21740019.}
} 
\date{\today}

\begin{document}

\setlength{\baselineskip}{15pt}
\maketitle

\begin{abstract}
 We establish ``higher depth'' analogues of regularized determinants
 due to Milnor for zeros of cuspidal automorphic $L$-functions
 of $\GL_d$ over a general number field.
 This is a generalization of the result of Deninger about the regularized determinant 
 for zeros of the Riemann zeta function. 

\begin{MSC}
 {\it Primary} 11F66
 {\it Secondary} 11M36,
\end{MSC} 
\end{abstract}

\section{Introduction}

 For a complex sequence $\bsym{a}=\{a_n\}_{n\in I}$, 
 the zeta regularized product of $\bsym{a}$ is defined by 
\[
 \regprod_{n\in I}a_n:=\exp\Bigl(-\frac{d}{ds}\z_{\bsym{a}}(s)\Bigl|_{s=0}\Bigr),
\]
 where $\z_{\bsym{a}}(s):=\sum_{n\in I}a_n^{-s}$ is the zeta function attached to
 $\bsym{a}$.
 Here, we let $a_n^{-s}:=\exp(-s\log{a_n})$ 
 with $\log$ being the principal branch of the logarithm  
 and assume that $\z_{\bsym{a}}(s)$ converges absolutely in some right half-plane,  
 admits a meromorphic continuation to a region containing the origin and is holomorphic
 at the origin. 
 This gives a generalization of the usual product. 
 In fact, if $I$ is a finite set,
 then $\regprod_{n\in I}a_n=\prod_{n\in I}a_n$.
 Let $\z(s):=\sum_{m\ge 1}m^{-s}$ be the Riemann zeta function
 and $\mathcal{R}$ the set of all non-trivial zeros of $\z(s)$.
 In 1992, Deninger \cite{Deninger1992} (see also \cite{SchroterSoule1994})
 obtained the following formula.
\begin{equation}
\label{for:Deninger}
 \Xi(z)
:=\regprod_{\r\in \mathcal{R}}\Bigl(\frac{z-\r}{2\pi}\Bigr)
=2^{-\frac{1}{2}}\frac{z(z-1)}{(2\pi)^2}\pi^{-\frac{z}{2}}\G\Bigl(\frac{z}{2}\Bigr)\z(z)
=2^{-\frac{1}{2}}\frac{z(z-1)}{(2\pi)^2}\Lambda(z),
\end{equation} 
 where $\G(z)$ is the gamma function and 
 $\Lambda(z):=\pi^{-\frac{z}{2}}\G(\frac{z}{2})\z(z)$
 is the complete Riemann zeta function.
 This is called the regularized product expression of the Riemann zeta function.
 The aim of the present paper is to define ``higher depth'' generalizations
 of the formula \eqref{for:Deninger} properly for $L$-functions attached to
 irreducible cuspidal automorphic representations of $\GL_d$ over a general number field and 
 calculate ``higher depth regularized products''
 for all non-trivial zeros of such $L$-functions explicitly.

 Now let us explain what the higher depth regularized products are. 
 In \cite{M},
 from the viewpoint of the Kubert identity
 which plays an important role in the study of Iwasawa theory,
 Milnor introduced the following 
 ``higher depth gamma function'' $\mG_r(z)$ for $r\in\bN$;
\[
 \mG_{r}(z):=\exp\Bigl(\frac{\p}{\p s}\z(s,z)\Bigl|_{s=1-r}\Bigr),
\]
 where $\z(s,z):=\sum^{\infty}_{m=0}(m+z)^{-s}$ is the Hurwitz zeta function.
 We call $\mG_r(z)$ a Milnor gamma function of depth $r$.
 He studied, for examples, special values, a Stirling type formula
 (that is, an asymptotic formula as $z\to+\infty$)
 and functional relations among them (see also \cite{KOW}). 
 We notice that the well-known Lerch formula
 $\exp(\frac{\p}{\p s}\z(s,z)\bigl|_{s=0})=\log{\frac{\G(z)}{\sqrt{2\pi}}}$
 implies that $\mG_{1}(z)=\frac{\G(z)}{\sqrt{2\pi}}$,
 whence $\mG_{r}(z)$ indeed gives a generalization of $\G(z)$.
 The Milnor gamma functions are also important in a sense that 
 there are many arithmetic objects
 whose associated zeta or $L$-function can be written as a polynomial in the Riemann zeta functions
 (e.g., the real analytic Eisenstein series (\cite{Bump1997})
 and the prehomogeneous vector space of symmetric matrices (\cite{IbukiyamaSaito1995})).
 In fact, in the case of studying their corresponding regularized product
 (equivalently, the derivative at the origin), 
 the Milnor gamma functions naturally come up.
 Based on the study of Milnor,
 we define a {\it higher depth regularized product} of the sequence $\bsym{a}$ by 
\[
 \underset{n\in I}{\regprod}^{[r]}a_n
:=\exp\Bigl(-\frac{d}{ds}\z_{\bsym{a}}(s)\Bigl|_{s=1-r}\Bigr).
\]
 Here, we further assume that $\z_{\bsym{a}}(s)$ 
 admits a meromorphic continuation to some region containing $s=1-r$
 and is holomorphic at the point.
 It is clear that $\regprod^{[1]}_{n\in I}a_n=\regprod_{n\in I}a_n$.
 Notice that, using this notation,
 it can be written as $\mG_r(z)^{-1}=\regprod^{[r]}_{n\ge 0}(n+z)$.

 Let $K$ be an algebraic number field of degree $n$,
 $\cO_K$ the ring of integers and $\bA_{K}$ the adele ring of $K$.
 Let $S_f$ and $S_{\infty}$ be the sets of all finite and infinite places of $K$, respectively.
 Write $S_{\infty}=S_{\bR}\sqcup S_{\bC}$, 
 where $S_{\bR}$ (resp. $S_{\bC}$) is the set of all real (resp. complex) places of $K$
 and put $r_1=\# S_{\bR}$ (resp. $r_2=\# S_{\bC}$).
 Let $\pi=\otimes_{v}\pi_v$ be an
 irreducible cuspidal automorphic representation of $\GL_d(\bA_{K})$.
 Then, from the general theory (see, e.g., \cite{{GodementJacquet1972}}), 
 we can define the $L$-function $L(s,\pi)$ by the Euler product  
\[
 L(s,\pi)
:=\prod_{v\in S_{f}}\prod^{d}_{j=1}\bigl(1-\alpha_{v,j}(\pi)q_v^{-s}\bigr)^{-1}
\qquad (\Re(s)>1),
\]
 where $q_v$ is the residue degree of the local field $K_v$
 with $K_v$ being the $v$-adic completion of $K$ at $v$
 and the complex number $\a_{v,j}(\pi)$ is determined
 by the local representation $\pi_v$ for each $v\in S_f$.
 Moreover, let $\L(s,\pi)$ be the complete $L$-function define by  
\[
 \L(s,\pi):=L_{\infty}(s,\pi)L(s,\pi), 
\]
 where $L_{\infty}(s,\pi)$ is defines by 
\[
 L_{\infty}(s,\pi)
:=\prod_{v\in S_{\infty}}\prod^{d}_{j=1}\G_v\bigl(s+\mu_{v,j}(\pi)\bigr).
\]
 Here, $\G_v(s)$ is defined by 
\[
 \G_v(s):=N_v(N_v\pi)^{-\frac{N_v s}{2}}\G\Bigl(\frac{N_v s}{2}\Bigl)
\]
 with $N_v:=1$ if $v\in S_{\bR}$ and $N_v:=2$ otherwise 
 and $\mu_{v,j}(\pi)$ is a complex number determined by $\pi_v$ for each $v\in S_{\infty}$
 (see Example~\ref{ex:d=1} for the case $d=1$).
 We note that $\Re(\mu_{v,j}(\pi))>-\frac{1}{2}$.
 It is known that
 $\L(s,\pi)$ can be continued analytically to the whole plane $\bC$
 except in the case $d=1$ and $\pi$ is the trivial character ${\bf 1}$ 
 (which implies that $L(s,\pi)$ is the Dedekind zeta function $\z_{K}(s)$ of $K$) 
 when $\L(s,\pi)$ has simple poles at $s=0$ and $s=1$. 
 Moreover, it satisfies the functional equation
\begin{equation}
\label{for:fe}
  N^{\frac{s}{2}}_{\pi}\L(s,\pi)
=\e_{\pi}N^{\frac{1-s}{2}}_{\pi}\L(1-s,\overline{\pi}),
\end{equation}
 where $N_{\pi}\ge 1$ is an integer called the conductor of $\pi$,
 $\e_{\pi}$ is the root number which has modulus $1$
 and $\overline{\pi}$ is the contragradient representation of $\pi$.

 For $T>0$, let $\mathcal{R}(T,\pi)$ be the set of non-trivial zeros of $L(s,\pi)$ 
 (that is, the zeros of $\L(s,\pi)$) with $|\Im(\rho)|<T$ and
 $\mathcal{R}(\pi):=\lim_{T\to \infty}\mathcal{R}(T,\pi)$.
 Then, the grand Riemann hypothesis asserts that $\Re(\rho)=\frac{1}{2}$ 
 for all $\rho\in\mathcal{R}(\pi)$.
 For $\Re(s)>1$ and $\Re(z)>1$, defined the function $\xi(s,z,\pi)$ by 
\[
 \xi(s,z,\pi)
:=\sum_{\rho\in \mathcal{R}(\pi)}\Bigl(\frac{z-\rho}{2\pi}\Bigr)^{-s}.
\]
 In this paper, the sum  
 $\sum_{\rho\in \mathcal{R}(\pi)}$ always means $\lim_{T\to\infty}\sum_{\rho\in\mathcal{R}(T,\pi)}$.
 The purpose of the present paper is to calculate the higher depth regularized product 
\begin{equation}
\label{def:Xi_r}
 \Xi_{r}(z,\pi)
:=\underset{\rho\in \mathcal{R}(\pi)}{\regprod}^{\!\!\!\!\![r]}\Bigl(\frac{z-\rho}{2\pi}\Bigr)
=\exp\Bigl(-\frac{\p}{\p s}\xi(s,z,\pi)\Bigl|_{s=1-r}\Bigr).
\end{equation}
 Remark that $\Xi_{r}(z,\pi)$ is well-defined 
 because, when $\Re(z)>1$, it will be shown that
 $\xi(s,z,\pi)$ admits a meromorphic continuation to $\bC$ as a function of $s$ and,
 in particular, is holomorphic at $s=1-r$ for all $r\in\bN$
 (Proposition~\ref{prop:merocon_xi}). 

 To state our main result,
 let us introduce a ``poly $L$-function'' $L_{r}(s,\pi)$. 
 Let $Li_r(z):=\sum^{\infty}_{m=1}\frac{z^m}{m^r}$ be the polylogarithm of degree $r$
 and $H_r(z):=\exp(-Li_r(z))$. 
 Then, the function $L_{r}(s,\pi)$ is defined by the following Euler product.
\begin{equation}
\label{def:poly-L}
 L_r(s,\pi)
:=\prod_{v\in S_f}\prod^{d}_{j=1}H_r\bigl(\a_{v,j}(\pi)q_v^{-s}\bigr)^{-(\log{q_v})^{1-r}}.
\end{equation}
 This infinite product converges absolutely for (at least) $\Re(s)>\frac{3}{2}-\frac{1}{d^2+1}$
 (Lemma~\ref{lem:absoconv}),
 whence $L_r(s,\pi)$ defines a holomorphic function in the region.
 It is easy to see that this is a poly-generalization of the $L$-function.
 Actually, when $r=1$, since $Li_1(z)=-\log{(1-z)}$ and hence $H_1(z)=1-z$,
 we have $L_1(s,\pi)=L(s,\pi)$.
 Several analytic properties of $L_r(s,\pi)$ are studied in Section~\ref{sec:GRH}.
 Moreover, let
\[
 L_{r,\infty}(s,\pi)
:=\prod_{v\in S_{\infty}}\prod^{d}_{j=1}\G_{r,v}\bigl(s+\mu_{v,j}(\pi)\bigr),
\]
 where   
\[
 \G_{r,v}(s):=\sqrt{2N_v}(N_v\pi)^{-\frac{1}{r}(N_v\pi)^{1-r}B_r(\frac{N_v
 s}{2})}\mG_r\Bigl(\frac{N_v s}{2}\Bigr)^{(N_v\pi)^{1-r}}
\]
 with $B_r(s)$ being the Bernoulli polynomial defined by the generating function
 $\frac{te^{ts}}{e^t-1}=\sum^{\infty}_{r=0}B_r(s)\frac{t^r}{r!}$. 
 Notice that, since $B_1(s)=s-\frac{1}{2}$ and $\mG_1(s)=\frac{\G(s)}{\sqrt{2\pi}}$,
 we have $\G_{1,v}(s)=\G_{v}(s)$ for all $v\in S_{\infty}$,
 whence $L_{1,\infty}(s,\pi)=L_{\infty}(s,\pi)$.
 Finally, let 
\[
 \L_r(s,\pi)
:=L_{r,\infty}(s,\pi)L_{r}(s,\pi)^{(-1)^{r-1}(r-1)!(2\pi)^{1-r}}.
\] 
 From the above observations, one has already known that $\L_1(s,\pi)=\L(s,\pi)$.  
 Then, the main result of the paper is as follows.

\begin{thm}
\label{thm:main}
 For $\Re(z)>\frac{3}{2}-\frac{1}{d^2+1}$, we have  
\begin{align}
\label{for:HDregCusp}
 \Xi_{r}(z,\pi)
&=\underset{\rho\in \mathcal{R}(\pi)}{\regprod}^{\!\!\!\!\![r]}\Bigl(\frac{z-\rho}{2\pi}\Bigr)
=2^{-\frac{1}{2}nd}\Biggl[\Bigl(\frac{z}{2\pi}\Bigr)^{(\frac{z}{2\pi})^{r-1}}
\Bigl(\frac{z-1}{2\pi}\Bigr)^{(\frac{z-1}{2\pi})^{r-1}}\Biggr]^{\d_{1,{\bf 1}}}\L_r(z,\pi),
\end{align}
 where $\d_{1,{\bf 1}}=\d_{1,{\bf 1}}(\pi):=1$ if $d=1$ and $\pi={\bf 1}$ and $0$ otherwise.
\end{thm}

 Letting $r=1$,
 one immediately obtains the following corollary,
 which in fact reproduces the result of Deninger
 by letting $K=\bQ$, $d=1$ and $\pi={\bf 1}$.

\begin{cor}
 We have
\begin{align}
\label{for:regCusp}
 \regprod_{\rho\in \mathcal{R}(\pi)}\Bigl(\frac{z-\rho}{2\pi}\Bigr)
=2^{-\frac{1}{2}nd}\Biggl[\frac{z(z-1)}{(2\pi)^2}\Biggr]^{\d_{1,{\bf 1}}}\L(z,\pi).
\end{align}
\qed
\end{cor}

 Notice that the expression \eqref{for:regCusp} is valid for all $z\in\bC$.
 See Corollary~\ref{cor:anaXi} for an analytic continuation of $\Xi_{r}(z,\pi)$ for $r\ge 2$.

 The organization of the paper is as follows.
 In Section~\ref{sec:explicit formula},
 we establish a Weil type explicit formula for $L(s,\pi)$.
 In Section~\ref{sec:proof}, using the explicit formula,
 we first give a meromorphic continuation of $\xi(s,z,\pi)$.
 Then, we prove the main result, that is, 
 calculate the derivatives of $\xi(s,z,\pi)$ at $s=1-r$.
 In Section~\ref{sec:GRH},
 we study several analytic properties of the poly $L$-function $L_r(s,\pi)$.
 In particular, we restate the grand Riemann hypothesis 
 in terms of analyticity of $L_2(s,\pi)$. 

 Throughout the present paper, 
 we denote by $\bC$, $\bR$ and $\bQ$ the fields of
 all complex, all real and all rational numbers, respectively.
 We also use the notation $\bZ$ and $\bN$ to denote the sets of
 all rational and all positive integers, respectively.  

\section{Explicit formulas}
\label{sec:explicit formula}

 To prove the main theorem,
 we need a Weil type explicit formula for $L(s,\pi)$.
 To obtain this,
 we start from the following lemma essentially established in \cite[{Proposition~3.1}]{Michel2002}.

\begin{lem}
 Let $Q>1$ and $\phi(x)$ be a function in the Schwartz space $\mathcal{S}(\bR)$ whose Fourier transform
 $\hat{\phi}(y):=\int^{\infty}_{-\infty}\phi(x)e^{-2\pi ixy}dx$ has compact support
 (in particular, $\phi$ can be extended as a smooth function on $\bC$). 
 Then, it holds that  
\begin{align}
\label{for:explicit formula 1}
 \sum_{\rho\in \mathcal{R}(\pi)}\phi\Bigl(\frac{\log{Q}}{2\pi i}\bigl(\rho-\frac{1}{2}\bigr)\Bigr)
&=\frac{\log{N_{\pi}}}{\log{Q}}\hat{\phi}(0)
+\Biggl[\phi\Bigl(\frac{\log{Q}}{4\pi i}\Bigr)+\phi\Bigl(-\frac{\log{Q}}{4\pi
 i}\Bigr)\Biggr]\d_{1,{\bf 1}}\\
& \ \ \ -\frac{1}{\log{Q}}\sum_{v\in S_{f}}\sum^{\infty}_{l=1}
\Biggl(\frac{\L_{\pi}(q_v^{l})}{q_v^{\frac{1}{2}l}}\hat{\phi}\Bigl(\frac{l\log{q_v}}{\log{Q}}\Bigr)+
\frac{\L_{\overline{\pi}}(q_v^{l})}{q_v^{\frac{1}{2}l}}\hat{\phi}\Bigl(-\frac{l\log{q_v}}{\log{Q}}\Bigr)\Biggr)\nonumber\\
& \ \ \ +\frac{1}{\log{Q}}\sum_{v\in S_{\infty}}\sum^{d}_{j=1}H_{v,j}(Q,\phi,\pi),\nonumber
\end{align}
 where $\L_{\pi}(q_v^{l}):=\log{q_v}\sum^{d}_{j=1}\a_{v,j}(\pi)^l$ and 
\[
 H_{v,j}(Q,\phi,\pi)
:=\int^{\infty}_{-\infty}\phi(t)
\Biggl(\frac{\G_v'}{\G_v}\Bigl(\frac{1}{2}+\mu_{v,j}(\pi)+\frac{2\pi it}{\log{Q}}\Bigr)
+\frac{\G_v'}{\G_v}\Bigl(\frac{1}{2}+\mu_{v,j}(\overline{\pi})-\frac{2\pi it}{\log{Q}}\Bigr)\Biggr)dt.
\]
\qed
\end{lem}

 Following the argument of Barner \cite{Barner1981}, from \eqref{for:explicit formula 1},
 we next establish a Weil type explicit formula.   
 For a function $F:\bR\to\bC$ of bounded variation
 (i.e., $V_{\bR}(F)<\infty$ where $V_{\bR}(F)$ is the total variation of $F$ on $\bR$),
 we define the function $\Phi_F(s)$ for $s\in\bC$ by
\begin{align*}
 \Phi_F(s)
:=\hat{F}\Bigl(-\frac{s-\frac{1}{2}}{2\pi i}\Bigr)
=\int^{\infty}_{-\infty}F(x)e^{(s-\frac{1}{2})x}dx.
\end{align*}
 Moreover, for $v\in S_{\infty}$ and $1\le j\le d$,
 let $F_{v,j}(x,\pi):=F(x)e^{-i\eta_{v,j}(\pi) x}$ and 
 $\widetilde{F}_{v,j}(x,\pi):=F_{v,j}(x,\pi)+F_{v,j}(-x,\pi)$
 where we write $\mu_{v,j}(\pi)=\xi_{v,j}(\pi)+i\eta_{v,j}(\pi)$
 with $\xi_{v,j}(\pi),\eta_{v,j}(\pi)\in\bR$. 

\begin{prop}
\label{prop:Explicitformula}
 Let $F:\bR\to\bC$ be a function of bounded variation.
 Suppose that $F$ satisfies the following conditions; 
\begin{enumerate}
 \item[$\mathrm{(a)}$]
 There is a positive constant $b$ such that $V_{\bR}(F(x)e^{(\frac{1}{2}+b)|x|})<\infty$. 
 \item[$\mathrm{(b)}$]
 $F$ is ``normalized'', that is, $2F(x)=F(x+0)+F(x-0)$ for $x\in\bR$.
\item[$\mathrm{(c)}$]
 For any $v\in S_{\infty}$ and $1\le j\le d$,
 $\widetilde{F}_{v,j}(x,\pi)=2F(0)+O(|x|)$ as $|x|\to 0$.
\end{enumerate}
 Then, it holds that 
\begin{align}
\label{for:explicitformula}
\sum_{\rho\in \mathcal{R}(\pi)}\Phi_F(\rho)
&=F(0)\log{\frac{N_{\pi}}{(2^{2r_2}\pi^{n})^d}}+\bigl(\Phi_F(0)+\Phi_F(1)\bigr)\d_{1,{\bf 1}}\\
&\ \ \ -\sum_{v\in S_f}\sum^{\infty}_{l=1}
\Biggl(\frac{\L_{\pi}(q_v^l)}{q_v^{\frac{1}{2}l}}F(l\log{q_v})
+\frac{\L_{\overline{\pi}}(q_v^l)}{q_v^{\frac{1}{2}l}}F(-l\log{q_v})\Biggr)
+\sum_{v\in S_{\infty}}\sum^{d}_{j=1}W_{v,j}(F,\pi),\nonumber
\end{align}
 where 
\[
  W_{v,j}(F,\pi)
:=\int^{\infty}_{0}\Biggl(\frac{N_{v}F(0)}{x}
-\widetilde{F}_{v,j}(x,\pi)
\frac{e^{(\frac{2}{N_v}-\frac{1}{2}-\xi_{v,j}(\pi))x}}{1-e^{-\frac{2}{N_v}x}}\Biggr)
 e^{-\frac{2}{N_{v}}x}dx.
\]
\end{prop}
\begin{proof}
 Letting $Q=e^{2\pi}$ and $\phi(x)=\hat{F}(-\frac{x}{2\pi})$ in \eqref{for:explicit formula 1}
 and noting that $\hat{\phi}(y)=2\pi F(2\pi y)$, we have  
\begin{align*}
\sum_{\rho\in \mathcal{R}(\pi)}\Phi_F(\rho)
&=F(0)\log{N_{\pi}}+\bigl(\Phi_F(0)+\Phi_F(1)\bigr)\d_{1,{\bf 1}}\\
&\ \ \ -\sum_{v\in S_f}\sum^{\infty}_{l=1}
\Biggl(\frac{\L_{\pi}(q_v^l)}{q_v^{\frac{1}{2}l}}F(l\log{q_v})
+\frac{\L_{\overline{\pi}(q_v^l)}}{q_v^{\frac{1}{2}l}}F(-l\log{q_v})\Biggr)
+\sum_{v\in S_{\infty}}\sum^{d}_{j=1}Y_{v,j}(F,\pi),
\end{align*}
 where 
\[
 Y_{v,j}(F,\pi)
:=\frac{1}{2\pi}\int^{\infty}_{-\infty}\hat{F}\Bigl(-\frac{t}{2\pi}\Bigr)
\Biggl(\frac{\G_v'}{\G_v}\Bigl(\frac{1}{2}+\mu_{v,j}(\pi)+it\Bigr)
+\frac{\G_v'}{\G_v}\Bigl(\frac{1}{2}+\mu_{v,j}(\overline{\pi})-it\Bigr)\Biggr)dt.
\]
 Notice that the conditions (a) and (b) guarantee the convergence
 of the infinite sum $\sum_{\rho\in \mathcal{R}(\pi)}\Phi_F(\rho)$ 
 (more precisely, see \cite{Barner1981}).
 Now, let us calculate the integral $Y_{v,j}(F,\pi)$.
 Since $\mu_{v,j}(\overline{\pi})=\overline{\mu_{v,j}(\pi)}=\xi_{v,j}(\pi)-i\eta_{v,j}(\pi)$,
 using the formula 
\[
 \frac{\G_v'}{\G_v}(s)
=-\frac{N_v}{2}\log{N_v \pi}+\frac{N_v}{2}\frac{\G'}{\G}\Bigl(\frac{N_v s}{2}\Bigr),
\]
 we have
\begin{align*}
 Y_{v,j}(F,\pi)
&=\frac{1}{2\pi}\int^{\infty}_{-\infty}
\Biggl[\hat{F}\Bigl(-\frac{t-\eta_{v,j}(\pi)}{2\pi}\Bigr)+\hat{F}\Bigl(\frac{t+\eta_{v,j}(\pi)}{2\pi}\Bigr)\Biggr]\frac{\G_v'}{\G_v}\Bigl(\frac{1}{2}+\xi_{v,j}(\pi)+it\Bigr)dt\\
&=\frac{1}{2\pi}\int^{\infty}_{-\infty}
\widetilde{F}_{v,j}(\cdot,\pi)^{\wedge}\Bigl(\frac{t}{2\pi}\Bigr)
\Biggl(-\frac{N_v}{2}\log{N_v \pi}+\frac{N_v}{2}\frac{\G'}{\G}\Bigl(\frac{N_v}{2}\bigl(\frac{1}{2}+\xi_{v,j}(\pi)+it\bigr)\Bigr)\Biggr)dt\\
&=F(0)\log{\frac{1}{(N_v \pi)^{N_v}}}
+\frac{N_v}{2}\frac{1}{2\pi}\int^{\infty}_{-\infty}
\widetilde{F}_{v,j}(\cdot,\pi)^{\wedge}\Bigl(\frac{t}{2\pi}\Bigr)
\frac{\G'}{\G}\Bigl(\frac{N_v}{2}\bigl(\frac{1}{2}+\xi_{v,j}(\pi)\bigr)+i\frac{N_v}{2}t\Bigr)dt.
\end{align*}
 Here, for $a,b>0$ and $G\in L^{1}(\bR)$
 satisfying $V_{\bR}(G)<\infty$ and $G(x)=G(0)+O(|x|)$ as $s\to 0$,
 the following formula was also established in \cite{Barner1981};
\[
 \frac{1}{2\pi}\int^{\infty}_{-\infty}\hat{G}\Bigl(\frac{t}{2\pi}\Bigr)
\frac{\G'}{\G}\Bigl(a+i\frac{t}{b}\Bigr)dt
=\int^{\infty}_{0}\Bigl(\frac{G(0)}{x}-\frac{be^{(1-a)bx}}{1-e^{-bx}}G(-x)\Bigr)e^{-bx}dx. 
\]
 Using this formula
 with $G=\widetilde{F}_{v,j}$,
 $a=\frac{N_v}{2}(\frac{1}{2}+\xi_{v,j}(\pi))$ and $b=\frac{2}{N_v}$
 (remark that, thanks to the assumption (c), we can indeed apply this formula),
 we eventually obtain  
\[
 Y_{v,j}(F,\pi)=F(0)\log{\frac{1}{(N_v \pi)^{N_v}}}+W_{v,j}(F,\pi).
\] 
 This completes the proof.
\end{proof}

\section{Proof of Theorem~\ref{thm:main}}
\label{sec:proof}

 Using the explicit formula \eqref{for:explicitformula},
 we first give a meromorphic continuation of $\xi(s,z,\pi)$.

\begin{prop}
\label{prop:merocon_xi}
 For $\Re(z)>1$, we have 
 \begin{align}
\label{for:xi2}
 \xi(s,z,\pi)
&=\Biggl[\Bigl(\frac{2\pi}{z}\Bigr)^s+\Bigl(\frac{2\pi}{z-1}\Bigr)^s\Biggr]\d_{1,{\bf 1}}
+\frac{(2\pi)^s}{2\pi i}\int_{C}t^{{-s}}\frac{L'}{L}(z-t,\pi)dt\\
&\ \ \ -\sum_{v\in S_{\infty}}\sum^{d}_{j=1}(N_v\pi)^{s}\z\Bigl(s,\frac{N_v(z+\mu_{v,j}(\pi))}{2}\Bigr),
\nonumber
\end{align}
 where $C$ is the contour consisting of the lower edge of the cut from $-\infty$ to $-\d$,
 the circle $t=\d e^{i\theta}$ for $-\pi\le \theta\le \pi$ and the upper edge of the cut from $-\d$ to
 $-\infty$. 
 This gives a meromorphic continuation of $\xi(s,z,\pi)$ as a function of $s$
 to the whole plane $\bC$ with a simple pole at $s=1$. 
\end{prop}
\begin{proof}
 For $\Re(z)>1$ and $\Re(s)>1$, define the function $F:\bR\to\bC$ by 
\[
 F(x):=
\begin{cases}
 x^{s-1}e^{-(z-\frac{1}{2})x} & (x\ge 0),\\
 0 & (x<0),
\end{cases}
\]
 which satisfies the conditions (a), (b) and (c) in Proposition~\ref{prop:Explicitformula}.
 We can easily see that  
\begin{align*}
 \Phi_F(w)&=\frac{\G(s)}{(z-w)^s}, \quad \textrm{whence} \quad  
 \Phi_F(0)=\frac{\G(s)}{z^s}, \ \Phi_F(1)=\frac{\G(s)}{(z-1)^s},
\end{align*}
 and
\begin{align*}
 W_{v,j}(F,\pi)
&=-\int^{\infty}_{0}x^{s-1}\frac{e^{-(z+\mu_{v,j}(\pi))x}}{1-e^{-\frac{2x}{N_{v}}}}dx\\
&=-\G(s)\Bigl(\frac{N_v}{2}\Bigr)^{s}\z\Bigl(s,\frac{N_v(z+\mu_{v,j}(\pi))}{2}\Bigr).
\end{align*}
 In the second equality,
 we have used the well-known formula 
\[
 \G(s)\z(s,z)=\int^{\infty}_{0}x^{s-1}\frac{e^{-zx}}{1-e^{-x}}dx \qquad (\Re(s)>1).
\]
 Therefore, the explicit formula \eqref{for:explicitformula} reads
\begin{align*}
 \xi(s,z,\pi)
&=\Biggl[\Bigl(\frac{2\pi}{z}\Bigr)^s+\Bigl(\frac{2\pi}{z-1}\Bigr)^s\Biggr]\d_{1,{\bf 1}}
-(2\pi)^s\sum_{v\in S_f}\sum^{\infty}_{l=1}\frac{\L_{\pi}(q_v^{l})}{q_v^{lz}}\frac{(l\log{q_v})^{s-1}}{\G(s)}\\
&\ \ \ -\sum_{v\in S_{\infty}}\sum^{d}_{j=1}(N_v\pi)^{s}\z\Bigl(s,\frac{N_v(z+\mu_{v,j}(\pi))}{2}\Bigr).
\end{align*}
 Moreover, using the equation 
\[
 \frac{a^{s-1}}{\G(s)}=\frac{1}{2\pi i}\int_{C}t^{{-s}}e^{at}dt \qquad (a>0),
\]
 we have 
\begin{align}
\label{for:integral}
-(2\pi)^s\sum_{v\in S_f}\sum^{\infty}_{l=1}\frac{\L_{\pi}(q_v^{l})}{q_v^{lz}}\frac{(l\log{q_v})^{s-1}}{\G(s)}
&=-\frac{(2\pi)^s}{2\pi i}\int_{C}t^{{-s}}\sum_{v\in
 S_f}\sum^{\infty}_{l=1}\L_{\pi}(q_v^{l})q_v^{-l(z-t)}dt\\
&=\frac{(2\pi)^s}{2\pi i}\int_{C}t^{{-s}}\frac{L'}{L}(z-t,\pi)dt.\nonumber
\end{align}
 Notice that the second equality follows from the expression 
\begin{equation}
\label{for:logderiLK}
 \frac{L'}{L}(s,\pi)
=-\sum_{v\in S_f}\sum^{\infty}_{l=1}\L_{\pi}(q_v^{l})q_v^{-ls},
\end{equation}
 which is obtained from the Euler product expression of $L(s,\pi)$.
 This shows the equation \eqref{for:xi2}. 
 By the same argument performed in \cite{Deninger1992},
 we see that the last integral in \eqref{for:integral} converges absolutely for all $s\in\bC$,
 whence it defines an entire function as a function of $s$.
 Therefore, the expression \eqref{for:xi2} 
 indeed gives a meromorphic continuation of $\xi(s,z,\pi)$ to $\bC$
 which has only a simple pole coming from the Hurwitz zeta function at $s=1$.
 This ends the proof.
\end{proof}

 Proposition~\ref{prop:merocon_xi} shows that 
 $\xi(s,z,\pi)$ is holomorphic at $s=1-r$ for all $r\in\bN$
 and hence, in particular, that the function $\Xi_r(z,\pi)$ is well-defined.
 We now give a proof of our main result.

\begin{proof}
[Proof of Theorem~\ref{thm:main}] 
 Let $\Re(z)>\frac{3}{2}-\frac{1}{d^2+1}$.
 Let us calculate the derivative of $\xi(s,z,\pi)$ at $s=1-r$ via the expression \eqref{for:xi2}. 
 Write $\xi(s,z,\pi)=A_0(s,z,\pi)+A_{f}(s,z,\pi)+A_{\infty}(s,z,\pi)$ where  
\begin{align*}
 A_0(s,z,\pi)
&:=\Biggl[\Bigl(\frac{2\pi}{z}\Bigr)^s+\Bigl(\frac{2\pi}{z-1}\Bigr)^s\Biggr]\d_{1,{\bf 1}},\\
 A_f(s,z,\pi)
&:=\frac{(2\pi)^s}{2\pi i}\int_{C}t^{{-s}}\frac{L'}{L}(z-t,\pi)dt,\\
 A_{\infty}(s,z,\pi)
&:=-\sum_{v\in S_{\infty}}\sum^{d}_{j=1}(N_v\pi)^{s}\z\Bigl(s,\frac{N_v(z+\mu_{v,j}(\pi))}{2}\Bigr).
\end{align*}
 At first, it is easy to see that 
\begin{align}
\label{for:deriA0} 
 \exp\Bigl(-\frac{\p}{\p s}A_0(s,z,\pi)\Bigl|_{s=1-r}\Bigr)
&=\Biggl[\Bigl(\frac{z}{2\pi}\Bigr)^{(\frac{z}{2\pi})^{r-1}}
\Bigl(\frac{z-1}{2\pi}\Bigr)^{(\frac{z-1}{2\pi})^{r-1}}\Biggr]^{\d_{1,{\bf 1}}}.
\end{align}
 We next calculate the derivative of $A_f(s,z,\pi)$ by the same manner in \cite{Deninger1992}.
 It is clear that  
\begin{align*}
 -\frac{\p}{\p s}A_f(s,z,\pi)\Bigl|_{s=1-r}
&=\frac{(2\pi)^{1-r}}{2\pi i}\int_{C}\frac{L'}{L}(z-t,\pi)
t^{{r-1}}\log{\frac{t}{2\pi}}dt.
\end{align*}
 Furthermore, since
\begin{align*}
 \frac{1}{2\pi i}\int_{C}\frac{L'}{L}(z-t,\pi)t^{{r-1}}\log{\frac{t}{2\pi}}dt
 &=\frac{1}{2\pi i}\int^{0}_{\infty}\frac{L'}{L}(z-xe^{-\pi i},\pi)(xe^{-\pi
 i})^{{r-1}}\log{\frac{xe^{-\pi i}}{2\pi}}e^{-\pi i}dx\\
&\ \ \ +\frac{1}{2\pi i}\int^{\infty}_{0}\frac{L'}{L}(z-xe^{\pi i},\pi)(xe^{\pi
 i})^{{r-1}}\log{\frac{xe^{\pi i}}{2\pi}}e^{\pi i}dx\\
 &=\frac{1}{2\pi i}\int^{\infty}_{0}\frac{L'}{L}(z+x,\pi)(-1)^{r-1}x^{{r-1}}\Bigl(\log{\frac{x}{2\pi}}-\pi
 i\Bigr)dx\\
&\ \ \ -\frac{1}{2\pi i}\int^{\infty}_{0}\frac{L'}{L}(z+x,\pi)(-1)^{r-1}x^{{r-1}}\Bigl(\log{\frac{x}{2\pi}}+\pi
 i\Bigr)dx\\
&=(-1)^{r}\int^{\infty}_{0}\frac{L'}{L}(z+x,\pi)x^{r-1}dx,
\end{align*}
 using the formula \eqref{for:logderiLK} again,
 we have 
\begin{align*}
 -\frac{\p}{\p s}A_f(s,z,\pi)\Bigl|_{s=1-r}
&=(-1)^{r-1}(2\pi)^{1-r}\sum_{v\in S_f}\sum^{\infty}_{l=1}\L_{\pi}(q_v^{l})q_v^{-lz}
\int^{\infty}_{0}x^{r-1}q_v^{-lx}dx\\
&=(-1)^{r-1}(r-1)!(2\pi)^{1-r}\sum_{v\in S_f}\sum^{\infty}_{l=1}\log{q_v}\sum^{d}_{j=1}\a_{v,j}(\pi)^l\frac{q_v^{-lz}}{(l\log{q_v})^r}\\
&=(-1)^{r-1}(r-1)!(2\pi)^{1-r}\sum_{v\in S_f}\sum^{d}_{j=1}(\log{q_v})^{1-r}Li_r\bigl(\a_{v,j}(\pi)q_v^{-z}\bigr)\\
&=(-1)^{r-1}(r-1)!(2\pi)^{1-r}\log{L_r(z,\pi)}.
\end{align*}
 This shows that
\begin{equation}
\label{for:deriAf}
 \exp\Bigl(-\frac{\p}{\p s}A_f(s,z,\pi)\Bigl|_{s=1-r}\Bigr)
=L_r(z,\pi)^{(-1)^{r-1}(r-1)!(2\pi)^{1-r}}. 
\end{equation}
 Finally, from the fact $\z(1-r,z)=-\frac{B_r(z)}{r}$
 and the definition $\log{\mG_r(z)}=\frac{\p}{\p s}\z(s,z)\bigl|_{s=1-r}$,
 we have 
\begin{align*}
 -\frac{d}{ds}&A_{\infty}(s,z,\pi)\Bigl|_{s=1-r}\\
&=\sum_{v\in S_{\infty}}\sum^{d}_{j=1}(N_v\pi)^{1-r}
\Biggl[-\frac{B_r(\frac{N_v(z+\mu_{v,j}(\pi))}{2})}{r}\log{(N_v\pi)}+\log\mG_{r}\Bigl(\frac{N_v(z+\mu_{v,j}(\pi))}{2}\Bigr)\Biggl],
\end{align*}
 whence 
\begin{align}
\label{for:deriAinfty}
 \exp\Bigl(-\frac{d}{ds}A_{\infty}(s,z,\pi)\Bigl|_{s=1-r}\Bigr)
&=\prod_{v\in S_{\infty}}\prod^{d}_{j=1}\frac{1}{\sqrt{2N_v}}\G_{r,v}\bigl(z+\mu_{v,j}(\pi)\bigr)\\
&=2^{-\frac{1}{2}nd}L_{r,\infty}(z,\pi).\nonumber
\end{align}
 Combining three equations \eqref{for:deriA0}, \eqref{for:deriAf} and \eqref{for:deriAinfty},
 we obtain the desired formula \eqref{for:HDregCusp}.
 This completes the proof of the theorem.
\end{proof} 



\begin{example}
\label{ex:d=1}
 Let us explain the case $d=1$ (and $r=1$) more precisely,
 that is, the case of Hecke $L$-functions (see \cite{Barner1981}).
 Let $\chi$ be a Hecke gr\"ossencharacter with the conductor $\ff$.
 It is known that there exists
 some $\vp_v=\vp_v(\chi)\in\bR$ with $\sum_{v\in S_{\infty}}N_v\vp_v=0$
 and $m_v=m_v(\chi)\in\bZ$ such that 
\[
 \c\bigl((\a)\bigr)
=\prod_{v\in S_{\infty}}|\a_v|^{-iN_v\vp_v}\Bigl(\frac{\a_v}{|\a_v|}\Bigr)^{m_v}
\quad (\a\in \cO_K\ \textrm{with}\ \a\equiv 1\ \mathrm{mod}^{\times}\,{\ff}),
\]
 where $\mathrm{mod}^{\times}$ indicates the multiplicative congruence, 
 $(\a)$ the principal ideal generated by $\a$ and 
 $\a_v$ the image of $\a$ of the embedding $K\hookrightarrow K_v$.
 In this case, we know that $\a_{\fp,1}(\chi)=\chi(\fp)$ for $\fp\in S_f$
 and $\m_{v,1}(\chi)=\frac{|m_{v}|}{N_v}+i\vp_{v}$ for $v\in S_{\infty}$.
 Namely,
\begin{align*}
 L(s,\chi)
&=\prod_{\fp\in S_f}\bigl(1-\chi(\fp)N(\fp)^{-s}\bigr)^{-1} \qquad (\Re(s)>1),\\
 L_{\infty}(s,\chi)
&=\prod_{v\in S_{\infty}}\G_v\Bigl(s+\frac{|m_{v}|}{N_v}+i\vp_{v}\Bigr), 
\end{align*}
 where $N$ denotes the absolute norm
 (that is, $N(\fp)=q_{\fp}=\#\cO_K/\fp$).
 We remark that $N_{\chi}$ in \eqref{for:fe} is given by
 $N_{\chi}=N(\ff)|d_K|$ where $d_K$ is the discriminant of $K$.
 Now, from \eqref{for:regCusp}, 
 letting $\L(s,\chi)=L_{\infty}(s,\chi)L(s,\chi)$,
 we have 
\[
 \regprod_{\rho\in \mathcal{R}(\chi)}\Bigl(\frac{z-\rho}{2\pi}\Bigr)
=2^{-\frac{1}{2}n}\Biggl[\frac{z(z-1)}{(2\pi)^2}\Biggr]^{\d_{{\bf 1}}}\L(z,\chi).
\]
 where $\d_{{\bf 1}}=\d_{{\bf 1}}(\chi):=1$ if $\chi={\bf 1}$ and $0$ otherwise.
\end{example}

\begin{remark}
  As analogues of Theorem~\ref{thm:main},
 ``higher depth determinants'' of the Laplacian on compact Riemann surfaces of genus
 $g\ge 2$ are investigated in \cite{KurokawaWakayamaYamasaki}
 (see also \cite{Yamasaki} for the corresponding results on higher dimensional spheres).
 We notice that these are defined like \eqref{def:Xi_r} but
 we employ the spectral zeta functions of the Laplacian on the Riemann surface instead of $\xi(s,z,\pi)$, 
 whence the determination of gamma factors is also involved.
\end{remark}

\begin{remark}
 It seems that the argument in this paper can be extended essentially
 for Dirichlet series in the so-called Selberg class, namely, 
 series satisfying the Ramanujan conjecture, has an Euler product expression, 
 admits an analytic continuation except for $s=1$
 and satisfies a functional equation (see, e.g., \cite{Steuding2007}). 
\end{remark}

\section{Analytic properties of the poly $L$-functions}
\label{sec:GRH}

 We first show that $L_r(s,\pi)$ converges absolutely in a right half-plane.

\begin{lem}
\label{lem:absoconv}
 The infinite product \eqref{def:poly-L} converges absolutely
 for $\Re(s)>\frac{3}{2}-\frac{1}{d^2+1}$. 
\end{lem}
\begin{proof}
 Using the equality (see \cite{LuoRudnickSarnak1999}, also \cite{IwaniecSarnak2000})
\[
 \bigl|\log_{q_v}|\a_{v,j}(\pi)|\bigr|<\frac{1}{2}-\frac{1}{d^2+1}, 
\]
 we have 
\begin{align*}
 \sum_{v\in S_f}
\Biggl|\log\Bigl(\prod^{d}_{j=1}H_r\bigl(\a_{v,j}(\pi)q_v^{-s}\bigr)^{-(\log{q_v})^{1-r}}\Bigr)\Biggr|
&\le  \sum_{v\in S_f}
\sum^{d}_{j=1}\sum^{\infty}_{m=1}(\log{q_v})^{1-r}\frac{1}{m^r}|\a_{v,j}(\pi)|^mq_v^{-m\Re(s)}\\
&\le \sum_{v\in S_f}
\sum^{d}_{j=1}\sum^{\infty}_{m=1}\frac{1}{m}q_v^{-m(\Re(s)-\frac{1}{2}+\frac{1}{d^2+1})}\\
&=d\log{\z_K\Bigl(\Re(s)-\frac{1}{2}+\frac{1}{d^2+1}\Bigr)}.
\end{align*}
 Hence the claim 
 follows from the fact that $\z_K(s)$ converges absolutely for $\Re(s)>1$.
\end{proof}

 Let $\Omega(\pi)$ be the set of all complex numbers which are not of the form $\rho-\lambda$ 
 where $\rho$ is a trivial or a non-trivial zero of $L(s,\pi)$ and $\lambda\ge 0$ or,
 if $d=1$ and $\pi={\bf 1}$, then $1-\lambda$ where $\lambda\ge 0$ (See Figure~1).
 Notice that the trivial zeros of $L(s,\pi)$ are
 of the form $\rho=-\mu_{v,j}(\pi)-\frac{2l}{N_v}$
 for $v\in S_{\infty}$, $1\le j\le d$ and $l\in\bZ_{\ge 0}$.
 We next give an analytic continuation of $L_r(s,\pi)$ to $\Omega(\pi)$ for $r\ge 2$.

\begin{figure}[htbp]
\label{fig:Omega}
 \begin{center}
 \psfrag{Re}{\small $\Re$}
 \psfrag{Im}{\small $\Im$}
 \psfrag{B}{\small $0$}
 \psfrag{C}{\small $1$}
 \psfrag{D}{\small $\frac{1}{2}$}
 \includegraphics[clip,width=60mm]{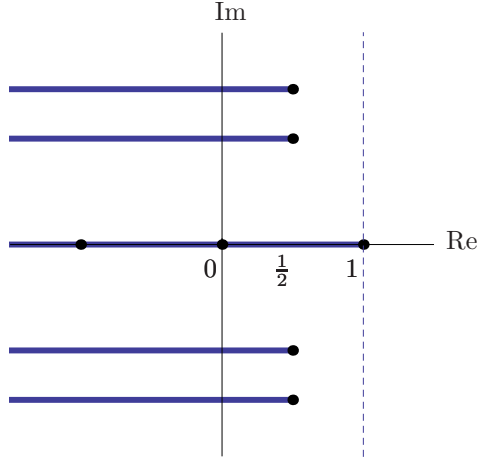}
 \caption{The region $\Omega(\pi)$ (if $d=1$ and $\pi={\bf 1}$).}
\end{center}
\end{figure}

\begin{lem}
 It holds that 
\begin{equation}
\label{for:ladder}
 \frac{d^{r-1}}{ds^{r-1}}\log{L_r(s,\pi)}
=(-1)^{r-1}\log{L(s,\pi)}
\qquad \bigl(\Re(s)>\frac{3}{2}-\frac{1}{d^2+1}\bigr).
\end{equation} 
\end{lem}
\begin{proof}
 The case $r=1$ is trivial.
 Assume that $r\ge 2$. Then, using the differential equation 
\[
 \frac{d}{dz}Li_r(z)=z^{-1}Li_{r-1}(z),
\]
 we have 
\begin{align*}
 \frac{d}{ds}\log{L_r(s,\pi)}
&=\frac{d}{ds}\sum_{v\in S_f}\sum^{d}_{j=1}(\log{q_v})^{1-r}Li_r\bigl(\a_{v,j}(\pi)q_v^{-s}\bigr)\\
&=-\sum_{v\in S_f}\sum^{d}_{j=1}(\log{q_v})^{1-(r-1)}Li_{r-1}\bigl(\a_{v,j}(\pi)q_v^{-s}\bigr)\\
&=-\log{L_{r-1}(s,\pi)}.
\end{align*}
 Therefore we inductively obtain the formula \eqref{for:ladder}.
\end{proof}

\begin{prop}
\label{prop:IteratedL}
 Fix $a\in\bC$ with $\Re(a)>\frac{3}{2}-\frac{1}{d^2+1}$.
 Then, for $r\ge 2$, we have  
\begin{align}
\label{for:IntRep}
 L_r(s,\pi)
=Q_{r}(s,a,\pi)\exp\Biggl(
 \underbrace{\int^{s}_{a}\int^{\xi_{r-1}}_{a}\cdots\int^{\xi_2}_{a}}_{r-1}
 \log L(\xi_1,\pi)d\xi_1\cdots d\xi_{r-1}\Biggl)^{(-1)^{r-1}}.
\end{align}
 Here $Q_r(s,a,\pi):=\prod^{r-2}_{k=0}L_{r-k}(a,\pi)^{\frac{(-1)^k}{k!}(s-a)^k}$.
 This gives an analytic continuation of $L_r(s,\pi)$ to the region $\Omega(\pi)$.
 In particular, $L_r(s,\pi)\ne 0$ for any $s\in \Omega(\pi)$.
\end{prop}
\begin{proof}
 The expression \eqref{for:IntRep} is obtained from \eqref{for:ladder} by induction on $r$.
 Moreover, since $\log L(\xi,\pi)$ is a (single-valued) holomorphic function in $\Omega(\pi)$,
 in the righthand-side of \eqref{for:IntRep}, one can move $s$ freely in $\Omega(\pi)$.
 Hence \eqref{for:IntRep} in fact gives 
 an analytic continuation of $L_r(s,\pi)$ to $\Omega(\pi)$.
\end{proof}

\begin{cor}
\label{cor:anaXi}
 The function
 $\Xi_{r}(z,\pi)$ with $r\ge 2$ admits an analytic continuation to the region $\Omega(\pi)$.
  Moreover, $\Xi_{r}(z,\pi)\ne 0$ for any $z\in \Omega(\pi)$.
\end{cor}
\begin{proof}
 We notice that, when $r\ge 2$,
 the Milnor gamma function $\mG_r(z)$ is a (single-valued) holomorphic function 
 in $\bC\setminus(-\infty,0]$ (see \cite{KOW}).
 Hence, the claim follows from the expression \eqref{for:HDregCusp} 
 together with Proposition~\ref{prop:IteratedL}.
\end{proof}

\begin{remark}
 Let
\[
 \widetilde{L}_{r}(s,\pi)
:=\prod_{v\in S_{f}}\prod^{d}_{j=1}H_r\bigl(\a_{v,j}(\pi)q_v^{-s}\bigr)^{-1}.
\]
 One can similarly prove that $\widetilde{L}_{r}(s,\pi)$
 converges absolutely for $\Re(s)>\frac{3}{2}-\frac{1}{d^2+1}$
 and see that  $\widetilde{L}_{1}(s,\pi)=L(s,\pi)$.
 However, it does not seem to have an analytic continuation to the whole plane $\bC$ when $r\ge 2$ again.
 Actually, in \cite{PolyEuler}, it was shown that
 $\widetilde{\zeta}_r(s):=\widetilde{L}_r(s,{\bf 1})$ in the case $K=\bQ$ and $d=1$ has an analytic
 continuation to the region $\Re(s)>0$ but has a natural boundary at the imaginary axis $\Re(s)=0$.
\end{remark}

 We finally show a relation between $L_r(s,\pi)$ and 
 the grand Riemann hypothesis for $L(s,\pi)$.
 Recall that the grand Riemann hypothesis asserts that $\Re(\rho)=\frac{1}{2}$ 
 for all $\rho\in\mathcal{R}(\pi)$.

\begin{cor}
 The grand Riemann hypothesis for $L(s,\pi)$ is equivalent to say that
 the function $(s-1)^{-\d_{1,{\bf 1}}(s-1)}L_2(s,\pi)$
 is a single-valued holomorphic function in $\Re(s)>\frac{1}{2}$.
\end{cor}
\begin{proof}
 Let $s\in\Omega(\pi)$ and $\Re(a)>\frac{3}{2}-\frac{1}{d^2+1}$.
 Then, from \eqref{for:IntRep} with $r=2$, we have 
\begin{equation}
\label{for:poly-dede2}
 L_2(s,\pi)
=L_2(a,\pi)\exp\Bigl(-\int^{s}_a\log{L(\xi,\pi)}d\xi\Bigr),
\end{equation}
 where the path is taken in $\Omega(\pi)$.  
 Notice that, since 
\[
 \int^{s}_a\log{\bigl((\xi-1)^{\d_{1,{\bf 1}}}\bigr)}d\xi
=\d_{1,{\bf 1}}\Bigl[(s-1)\log{(s-1)}-s-\bigl((a-1)\log{(a-1)}-a\bigr)\Bigr],
\]
 we have 
\begin{equation}
\label{for:log-formula}
 e^{\d_{1,{\bf 1}}s}(s-1)^{-\d_{1,{\bf 1}}(s-1)}
=e^{\d_{1,{\bf 1}}a}(a-1)^{-\d_{1,{\bf 1}}(a-1)}
\exp\Bigl(-\int^{s}_a\log{\bigl((\xi-1)^{\d_{1,{\bf 1}}}\bigr)}d\xi\Bigr).
\end{equation}
 Therefore, from \eqref{for:poly-dede2} and \eqref{for:log-formula},
 we have
\begin{align*}
 e^{\d_{1,{\bf 1}}s}(s-&1)^{-\d_{1,{\bf 1}}(s-1)}L_2(s,\pi)\\
&=e^{\d_{1,{\bf 1}}a}(a-1)^{-\d_{1,{\bf 1}}(a-1)}L_2(a,\pi)
 \exp\Bigl(-\int^{s}_a\log{\bigl((\xi-1)^{\d_{1,{\bf 1}}}L(\xi,\pi)\bigr)}d\xi\Bigr).
\end{align*}
 Now the statement follows from the fact that
 $(\xi-1)^{\d_{1,{\bf 1}}}L(\xi,\pi)$ is holomorphic at $\xi=1$.
\end{proof}

\begin{Acknowledgement}
 The authors would like to thank the referee and the editor for
 their valuable comments and suggestions which help us to improve the manuscript.
 The authors also thank Kazufumi Kimoto and Miki Hirano for their interest to this study.
\end{Acknowledgement}



\bigskip

\noindent
\textsc{Masato WAKAYAMA}\\
 Faculty of Mathematics, Kyushu University,\\
 Motooka, Nishiku, Fukuoka, 819-0395, JAPAN.\\
\texttt{wakayama@math.kyushu-u.ac.jp}\\

\noindent
\textsc{Yoshinori YAMASAKI}\\
 Graduate School of Science and Engineering, Ehime University,\\
 Bunkyo-cho, Matsuyama, 790-8577 JAPAN.\\
 \texttt{yamasaki@math.sci.ehime-u.ac.jp}

\end{document}